\documentclass[11pt]{amsart}

\usepackage{amsmath,amsthm,amssymb, amscd}
\usepackage[margin=1in]{geometry}
\usepackage[bookmarks]{hyperref}
\usepackage{verbatim}

\newtheorem{thm}{Theorem}[subsection]
\newtheorem{prop}{Proposition}[subsection]
\newtheorem{conj}{Conjecture}[subsection]
\newtheorem{cor}{Corollary}[subsection]

\theoremstyle{definition}

\newcommand{\C}{{\mathbb C}}
\newcommand{\Z}{{\mathbb Z}}

\newcommand{\Ext}{\operatorname{Ext}}

\newcommand{\Hom}{\operatorname{Hom}}
\newcommand{\0}{\bar 0}
\newcommand{\1}{\bar 1}

\newcommand{\g}{\ensuremath{\mathfrak{g}}}
\newcommand{\h}{\ensuremath{\mathfrak{h}}}
\newcommand{\e}{\ensuremath{\mathfrak{e}}}
\newcommand{\f}{\ensuremath{\mathfrak{f}}}

\newcommand{\res}{\ensuremath{\operatorname{res}}}

\numberwithin{equation}{subsection}
\numberwithin{table}{subsection}

\begin{document}

\title[Detecting Cohomology for Lie Superalgebras]
{\bf Detecting Cohomology for Lie Superalgebras}

\author{\sc Gustav I. Lehrer}
\address
{School of Mathematics and Statistics F07\\
University of Sydney NSW 2006\\
Australia }
\thanks{Research of the first and third authors was supported in part by ARC grant DP0772870} \email{gustav.lehrer@sydney.edu.au}

\author{\sc Daniel K. Nakano}
\address
{Department of Mathematics\\ University of Georgia \\
Athens\\ GA~30602, USA}
\thanks{Research of the second author was supported in part by NSF
grant  DMS-1002135} \email{nakano@math.uga.edu}

\author{\sc Ruibin Zhang}
\address
{School of Mathematics and Statistics F07\\
University of Sydney NSW 2006\\
Australia }
\email{ruibin.zhang@sydney.edu.au}

\date{October 2010}

\maketitle

\begin{abstract} In this paper we use invariant theory to develop the notion of cohomological detection for Type I classical Lie superalgebras. 
In particular we show that the cohomology with coefficients in an arbitrary module can be detected on smaller subalgebras. 
These results are used later to affirmatively answer questions, which were originally posed in \cite{BKN1} and \cite{BaKN}, about realizing support varieties for Lie superalgebras via 
rank varieties constructed for the smaller detecting subalgebras. 
\end{abstract} 

\section{Introduction}

\subsection{} For finite groups there are well known local-global principles which enable
the study of their representation theory and cohomology via that of proper subgroups. For example if $G$ is a
finite group, $k$ is field of characteristic $p>0$, and $P$ is a $p$-Sylow subgroup of $G$, then
the restriction map induces an embedding $\text{res}:\text{Ext}^{\bullet}_{kG}(M,N)\hookrightarrow
\text{Ext}^{\bullet}_{kP}(M,N)$ where $kG$ (resp. $kP$) is the group algebra of $G$ (resp. $P$).
We therefore say that the cohomology is ``detected by'' the $p$-Sylow subgroups.
Another collection of subgroups which detect the cohomology is the set ${\mathcal E}$
of elementary abelian $p$-subgroups.
Here the restriction map
induces an inseparable isogeny ($F$-isomorphism):
$$\text{H}^{\bullet}(G,k)\rightarrow \lim_{E\in {\mathcal E}}\text{H}^{\bullet}(E,k).$$
Moreover, a cohomology class $\zeta\in \text{Ext}^{\bullet}_{kG}(M,M)$ is nilpotent if and only if
$\text{res}(\zeta)$ is nilpotent in $\text{Ext}^{\bullet}_{kE}(M,M)$ for every $E\in {\mathcal E}$.

Such cohomological ``detection theorems'' may be used to deduce properties
of support varieties of finite groups. Let $M$ be a finite-dimensional module for $kG$ and
write ${\mathcal V}_{kG}(M)$ (resp. ${\mathcal V}_{kP}(M)$) for its support variety over $kG$ (resp.
$kP$). The restriction map induces a morphism of algebraic varieties
$\text{res}^{*}:{\mathcal V}_{kP}(M)\rightarrow {\mathcal V}_{kG}(M)$
which is finite to one. Moreover,
$${\mathcal V}_{kG}(M)=\bigcup_{E\in {\mathcal E}}\text{res}^{*}({\mathcal V}_{kE}(M)).$$
Detectability on small subgroups is a rather special feature of modular group algebras. For
other finite-dimensional cocommutative Hopf algebras, like the restricted enveloping algebra of
a restricted Lie algebra, one can define cohomology and support varieties, but cohomology is rarely
detected on a {\em finite} set of smaller (proper) subalgebras.

\subsection{} We shall be concerned in this work with the 
representation theory of finite-dimensional classical 
Lie superalgebras ${\mathfrak g}={\mathfrak g}_{\0}
\oplus {\mathfrak g}_{\1}$ over the complex numbers, which has strong 
analogies with the finite group case. Boe, Kujawa, and Nakano \cite{BKN1}
recently initiated the study of local-global principles in the setting of Lie superalgebras.
Using natural properties of the action of the reductive 
group $G_{\0}$ (where $\text{Lie }G_{\0}=\g_{\0})$ on
${\mathfrak g}_{\1}$, they proved the existence of two types of detecting subalgebras,
${\mathfrak f}={\mathfrak f}_{\0} \oplus {\mathfrak f}_{\1}$ and ${\mathfrak e}=
{\mathfrak e}_{\0} \oplus {\mathfrak e}_{\1}$. These subalgebras were used to study
representations in the category ${\mathcal F}_{(\g,\g_{\0})}$ of
finite-dimensional $\g$-modules which are semisimple over $\g_{\0}$.
The present work is a continuation of that program.

In this situation, the
restriction maps induce isomorphisms:
$$\text{H}^{\bullet}(\g,\g_{\0},{\mathbb C})\cong
\text{H}^{\bullet}(\f,\f_{\0},{\mathbb C})^{N/N_{0}}\cong
\text{H}^{\bullet}(\e,\e_{\0},{\mathbb C})^W$$
where $N/N_{0}$ is a reductive group and $W=W(\e)$ is a finite pseudoreflection group. These
relative cohomology rings may be identified with the invariant ring
$S^{\bullet}({\mathfrak g}_{\1}^{*})^{G_{\0}}$, where $S^{\bullet}$ denotes the symmetric
algebra, and so are
finitely generated. This property was used to construct support varieties for
$M$ in ${\mathcal F}_{(\g,\g_{\0})}$. The restriction maps in cohomology induce embeddings
of support varieties:
\begin{equation}\label{supportemb}
{\mathcal V}_{(\e,\e_{\0})}(M)/W\hookrightarrow
{\mathcal V}_{(\f,\f_{\0})}(M)/(N/N_{0})\hookrightarrow
{\mathcal V}_{(\g,\g_{\0})}(M).
\end{equation}
It was suspected that these embeddings are in fact isomorphisms, but when these varieties were introduced
there was no reasonable theory of cohomological detection for arbitrary modules in 
${\mathcal F}_{(\g,\g_{\0})}$.

An analogous theory has been developed for the Lie superalgebras $W(n)$ and $S(n)$ of Cartan type
in \cite{Ba, BaKN}. These Lie superalgebras are ${\mathbb Z}$-graded and detecting subalgebras
were constructed using the reductive group corresponding to the zero component of the Lie superalgebra.

\subsection{} The main goal of this paper is to develop the remarkable
theory of cohomological detection for arbitrary Type I classical Lie superalgebras and the
Lie superalgebras of Cartan type, $W(n)$ and $S(n)$. The class of classical Lie superalgebras under 
consideration
includes the general linear Lie superalgebra ${\mathfrak g}=\mathfrak{gl}(m|n)$.

In Section 2, we review the fundamental definitions of classical Lie superalgebras and the
constructions of the detecting subalgebras $\f$ and $\e$. We also indicate how detecting subalgebras are
constructed for $W(n)$ and $S(n)$. In the following section, we
prove that for Type I classical Lie superalgebras, the restriction map
\begin{equation}\label{E:injectf}
\text{res}:\text{H}^{n}(\g,\g_{\0},M)\hookrightarrow \text{H}^{n}(\f,\f_{\0},M)
\end{equation}
is injective for all $n\geq 0$ and $M$ in ${\mathcal F}_{(\g,\g_{\0})}$.
The same arguments may be used to verify cohomological embedding results for
$W(n)$ and $S(n)$.

In Section 4, we show by example that (\ref{E:injectf}) does not hold when $\f$ is
replaced by $\e$. However, one can describe a relationship between support varieties
of $\f$ and $\e$ by using an auxiliary sub Lie superalgebra $\overline{\f}$.
These cohomological embedding results are then applied to the theory of
support varieties, and used to prove that the embeddings given in (\ref{supportemb})
are indeed isomorphisms of varieties. One consequence of this result is the concrete
realization of the support variety ${\mathcal V}_{(\g,\g_{\0})}(M)$ as a quotient of
a rank variety over $\e_{\1}$ by the finite (pseudo) reflection group $W=W(\e)$
(cf. Theorem~\ref{T:supportprop}(a)).
Our results indicate the importance of the subalgebras $\f$ and $\e$ for the 
theory of classical Lie superalgebras.
Finally, in Section 5, we apply our results to show that these support varieties can be viewed as support
data as defined by Balmer \cite{Bal}. We also indicate how the support theory fits 
into the classical combinatorial
notion of atypicality as defined by Kac, Wakimoto and Serganova.

This paper was completed while the second author was visiting the University of Sydney during
Spring 2010. The second author would like to express his gratitude for the hospitality and
support of the University of Sydney and the University of New South Wales during
his residence in Australia.

\section{Detecting subalgebras}

\subsection{Notation: } We will use and summarize the conventions developed in
\cite{BKN1, BKN2, BKN3}. For more details we refer the reader to \cite[Section 2]{BKN1}.

Throughout this paper, let ${\mathfrak g}$ be a Lie superalgebra over the complex numbers ${\mathbb C}$.
In particular, ${\mathfrak g}={\mathfrak g}_{\0}\oplus {\mathfrak g}_{\1}$ is a $\Z_{2}$-graded vector space
with a supercommutator $[\;,\;]:\g \otimes \g \to \g$. A finite dimensional Lie superalgebra
${\mathfrak g}$ is called \emph{classical} if there is a connected reductive algebraic group $G_{\0}$ such
that $\operatorname{Lie}(G_{\0})=\g_{\0},$ and the action of $G_{\0}$ on $\g_{\1}$ differentiates to
the adjoint action of $\g_{\0}$ on $\g_{\1}.$  We say that ${\mathfrak g}$ is a \emph{basic classical}
Lie superalgebra if it is a classical Lie superalgebra with a nondegenerate invariant supersymmetric
even bilinear form.

Let $U(\g)$ be the universal enveloping superalgebra of ${\mathfrak g}$. We will be interested in
supermodules which are $\Z_{2}$-graded left $U(\g)$-modules. If $M$ and $N$ are ${\mathfrak g}$-supermodules
one can use the antipode and coproduct of $U({\mathfrak g})$ to define a ${\mathfrak g}$-supermodule
structure on the dual $M^{*}$ and the tensor product $M\otimes N$. For the remainder of the
paper the term ${\mathfrak g}$-module will mean a ${\mathfrak g}$-supermodule. In order to
apply homological algebra techniques, we will restrict ourselves to the
\emph{underlying even category}, consisting of ${\mathfrak g}$-modules with the degree preserving morphisms.
In this paper we will study homological properties of the category
${\mathcal F}_{({\mathfrak g},{\mathfrak g}_{\0})}$ which is the full subcategory of
finite dimensional ${\mathfrak g}$-modules which are finitely semisimple over ${\mathfrak g}_{\0}$ (a
${\mathfrak g}_{\0}$-module is \emph{finitely semisimple} if it decomposes into a direct sum of
finite dimensional simple ${\mathfrak g}_{\0}$-modules).

The category ${\mathcal F}:={\mathcal F}_{(\g,\g_{\0})}$ has enough injective
(and projective) modules. In fact, ${\mathcal F}_{(\g,\g_{\0})}$
is a Frobenius category (i.e., where injectivity is equivalent to projectivity) \cite{BKN3}.
Given $M, N$ in ${\mathcal F}$, let $\Ext_{\mathcal{F}}^{d}(M,N)$ be the degree $d$
extensions between $N$ and $M$. In practice, there is a concrete realization for these
extension groups via the relative Lie superalgebra cohomology for the pair
$({\mathfrak g},{\mathfrak g}_{\0})$:
$$\text{Ext}_{\mathcal{F}}^{d}(M,N)\cong \text{H}^{d}({\mathfrak g},{\mathfrak g}_{\0}; M^{*}\otimes N).$$
The relative cohomology can be computed using an explicit complex (cf. \cite[Section 2.3]{BKN1}).
Moreover, the cohomology ring
$$
R:=\text{H}^{\bullet}({\mathfrak g},{\mathfrak g}_{\0}; {\mathbb C})=
S^{\bullet}({\mathfrak g}_{\1}^*)^{G_{\0}}.
$$
Since $G_{\0}$ is reductive it follows that $R$ is finitely generated.

\subsection{Classical Lie superalgebras: }  In this section we review the construction of the two
families of (cohomological) detecting subalgebras for classical Lie superalgebras 
defined in \cite[Section 3,4]{BKN1}
using the invariant theory of $G_{\0}$ on ${\mathfrak g}_{\1}$.

First we consider the case when $G_{\0}$ has a {\em stable} action on ${\mathfrak g}_{\1}$
(cf. \cite[Section 3.2]{BKN1}). That is, there is an open dense subset of $\g_{\1}$ consisting of
semisimple points. Recall that a point $x\in {\mathfrak g}_{\1}$ is called
{\em semisimple} if the orbit $G_{\0}\cdot x$ is closed in ${\mathfrak g}_{\1}$.

Let $x_{0}$ be a generic point in ${\mathfrak g}_{\1}$; that is, $x_0$ is semisimple and
regular, in the sense that its stablizer has minimal dimension. 
Let $H=\text{Stab}_{G_{\0}}x_0$ and $N:=N_{G_{\0}}(H)$.
In order to construct a detecting subalgebra, we let
${\mathfrak f}_{\1}={\mathfrak g}_{\1}^{H}$, ${\mathfrak f}_{\0}=\text{Lie }N$, and set
$${\mathfrak f}={\mathfrak f}_{\0}\oplus {\mathfrak f}_{\1}.$$
Then ${\mathfrak f}$ is a classical Lie superalgebra and a sub Lie superalgebra
of ${\mathfrak g}$. The stability of the action of $G_{\0}$ on ${\mathfrak g}_{\1}$ 
implies the following properties.

\begin{itemize}
\item[(2.2.1)] The restriction homomorphism $S(\g_{\1}^{*}) \to S(\f_{\1}^{*})$ induces an isomorphism
\[
\operatorname{res}: \text{H}^{\bullet}({\mathfrak g},{\mathfrak g}_{\0},{\mathbb C}) \rightarrow
\text{H}^{\bullet}({\mathfrak f},{\mathfrak f}_{\0},{\mathbb C})^{N/N_{0}}.
\]
Here $N_{0}$ is the connected component of the identity in $N$.
\item[(2.2.2)] The set $G_{\0}\cdot {\mathfrak f}_{\1}$ is dense in $\g_{\1}.$
\end{itemize}

Next we recall the notion of {\em polar action} introduced by Dadok and Kac \cite{dadokkac}.
Let $v \in {\mathfrak g}_{\1}$ be a semisimple element, and set
\begin{equation*}
{\mathfrak e}_{v}=\left\{x \in \g_{\1} \:\vert\:  \g_{\0}.x \subseteq \g_{\0}.v \right\},
\end{equation*}
where $\g_{\0}$ is the Lie algebra of $G_{\0}$. The action of $G_{\0}$ on ${\mathfrak g}_{\1}$
is called {\em polar} if for some semisimple element $v \in {\mathfrak g}_{\1}$ we have
$\operatorname{dim} {\mathfrak e}_{v} = \operatorname{dim} S({\mathfrak g}_{\1}^{*})^{G_{\0}},$
where $\operatorname{dim} S({\mathfrak g}_{\1}^{*})^{G_{\0}}$ is the Krull dimension of this ring. 
The vector space ${\mathfrak e}_{v}$ is called a \emph{Cartan subspace}; let
${\mathfrak e}_{\1}$ denote a fixed choice of a Cartan subspace.

If the action of $G_{\0}$ on $\g_{\1}$ is both stable and polar, we
can further assume
\begin{equation*}\label{E:containments}
x_{0} \in \e_{x_{0}}=\e_{\1} \subseteq {\mathfrak f}_{\1},
\end{equation*}
where $x_{0}$ and ${\mathfrak f}_{\1}$ are as above. Furthermore,
the Cartan subspace is unique up to conjugation by $G_{\0}$ (cf. \cite[Theorem 2.3]{dadokkac}).
Set $\e_{\0}=\text{Lie}(H)$. Then the detecting subalgebra ${\mathfrak e}$ is the classical
Lie sub-superalgebra of ${\mathfrak g}$ defined by:
$${\mathfrak e}={\mathfrak e}_{\0}\oplus {\mathfrak e}_{\1}.$$

Assume $\g$ is a classical Lie superalgebra where the action of $G_{\0}$ is both
stable and polar on ${\mathfrak g}_{\1}$. Then by \cite[Theorem 3.3.1]{BKN1} we have
the following two facts.
\begin{itemize}
\item[(2.2.3)] The restriction homomorphism $S(\g_{\1}^{*}) \to S(\e_{\1}^{*})$ induces an isomorphism
$$
\operatorname{res}: \text{H}^{\bullet}({\mathfrak g},{\mathfrak g}_{\0},{\mathbb C}) \rightarrow
\text{H}^{\bullet}({\mathfrak e},{\mathfrak e}_{\0},{\mathbb C})^{W},
$$
where $W=W({\mathfrak e})$ is a finite pseudoreflection group. In particular, $R$ is a polynomial algebra.
\vskip .15cm
\item[(2.2.4)] The set $G_{\0}.\e_{\1}$ is dense in $\g_{\1}$.
\end{itemize}

For the Lie superalgebras ${\mathfrak g}=W(n)$ and $S(n)$ detecting families (analogous 
to the $\f$'s) were also
constructed using stable actions. We will describe a basis for these subalgebras below.

\subsection{ Type I Lie superalgebras:} A Lie superalgebra is said to be of {\em Type I} if
it admits a $\Z$-grading ${\mathfrak g}=\g_{-1}\oplus {\g}_{0}\oplus {\g}_{1}$ concentrated
in degrees $-1,$ $0,$ and $1$ with ${\mathfrak g}_{\0}={\mathfrak g}_{0}$ and
${\mathfrak g}_{\1}=\g_{-1}\oplus {\g}_{1}$. Otherwise, ${\mathfrak g}$ is of Type II.
Examples of Type I Lie superalgebras include: $\mathfrak{gl}(m|n)$ and simple Lie superalgebras of
types $A(m,n)$, $C(n)$ and $P(n)$.

The simple modules for ${\mathfrak g}$, a Type I classical Lie superalgebra, can be constructed in the
following way. Let ${\mathfrak t}$ be a Cartan subalgebra of ${\mathfrak g}_{\0}$ and
$X^{+}_{0} \subseteq {\mathfrak t}^{*}$ be the set of dominant integral weights (with
respect to a fixed Borel subalgebra). For $\lambda \in X^{+}_{0},$ let $L_{0}(\lambda)$
be the simple finite dimensional ${\mathfrak g}_{\0}$-module of highest weight $\lambda$. Set
$$
{\mathfrak g}^{+}= {\mathfrak g}_{0} \oplus {\mathfrak g}_{1}
\qquad \text{and} \qquad  {\mathfrak g}^{-}= {\mathfrak g}_{0} \oplus {\mathfrak g}_{-1} .
$$
Since ${\mathfrak g}$ is a Type I Lie superalgebra ${\mathfrak g}_{\pm 1}$ is an abelian
ideal of ${\mathfrak g}^{\pm}$. We can therefore
view $L_{0}(\lambda)$ as a simple ${\mathfrak g}^{\pm}$-module via inflation.

For each $\lambda\in X^{+}_{0}$, we construct the \emph{Kac module} and the \emph{dual Kac module}
by using the tensor product and the Hom-space in the following way:
$$
K(\lambda)=U({\mathfrak g})\otimes_{U({\mathfrak g}^{+})} L_{0}(\lambda)
\qquad \text{and} \qquad K^{-}(\lambda) =
\Hom_{U({\mathfrak g}^{-})}\left(U({\mathfrak g}), L_{0}(\lambda) \right).
$$

The module $K(\lambda)$ has a unique maximal submodule. The head of $K(\lambda)$ is the simple
finite dimensional ${\mathfrak g}$-module $L(\lambda)$. Then $\{L(\lambda):\ \lambda \in X_{0}^{+}\}$
is a complete set of non-isomorphic simple modules in $\mathcal{F}_{(\g,\g_{\0})}$.  Let $P(\lambda)$
(resp.\ $I(\lambda)$) denote the projective cover 
(resp.\ injective hull) in $\mathcal{F}_{(\g,\g_{\0})}$ for the
simple ${\mathfrak g}$-module $L(\lambda)$. These are all finite-dimensional. Moreover, the projective
covers admit filtrations with sections being Kac modules and the injective hulls
have filtrations whose sections are dual Kac modules. These filtrations also respect the
ordering on weights and thus ${\mathcal F}_{(\g,\g_{\0})}$ is a highest 
weight category (cf. \cite[Section 3]{BKN3}) as defined in
\cite{CPS}.

\subsection{General Linear Superalgebra:} The prototypical example of a Type I classical Lie superalgebra
admitting both a stable and polar action of $G_{\0}$ on $\g_{\1}$ is $\g=\mathfrak{gl}(m|n)$,
which as a vector space is the set of $m+n$ by $m+n$ matrices. As
basis one may take the matrix units $E_{i,j}$ where $1\leq i,j \leq m+n$. The degree zero component
${\mathfrak g}_{\0}$ is the span of $E_{i,j}$ where $1\leq i,j  \leq m$ or $m+1 \leq i,j \leq m+n$.
As a Lie algebra ${\mathfrak g}_{\0}\cong \mathfrak{gl}(m)\times \mathfrak{gl}(n)$, and
the corresponding reductive group is $G_{\0}=GL(m)\times GL(n)$. Note that $G_{\0}$ acts on
$\mathfrak g_{\1}$ via the adjoint representation. A basis for ${\mathfrak g}_{\1}$ is
given by the $E_{i,j}$ such that $m+1\leq i \leq m+n$ and $1\leq j \leq n$ or
$1\leq i \leq m$ and $m+1 \leq j \leq m+n$.

Observe that ${\mathfrak g}_{\1}={\mathfrak g}_{-1}\oplus {\mathfrak g}_{1}$ where ${\mathfrak g}_{-1}$
(resp. ${\mathfrak g}_{1}$) consists of the lower triangular matrices (resp. upper triangular matrices)
in ${\mathfrak g}_{\1}$. The action of $G_{\0}$ on $\g_{-1}$ is given by $(A,B).X=BXA^{-1}$ so
the orbits are the matrices of a given rank in $\g_{-1}$. By results from
elementary linear algebra, $G_{\0}\cdot \f_{\1}$ and $G_{\0}\cdot \e_{\0}$ are dense in $\g_{\1}$.

For simplicity of exposition, let us assume that
$m=n=r$. With an appropriate choice of $x_{0}$ the detecting
subalgebras have the following descriptions. The detecting subalgebra $\f=\f_{\0}\oplus \f_{\1}$ where
$\f_{\1}$ is the span of $\{E_{i,i+r}:\  i=1,2,\dots,r\}\cup \{E_{i+r,i}:\  i=1,2,\dots,r\}$
and $\f_{\0}$ is the span of $[\f_{1},\f_{\1}]$. Here $H\cong T^{r}$ where $T$ is a one-dimensional
torus, and $N\cong W\ltimes T^{r}$ where $W=\Sigma_{r}\ltimes ({\mathbb Z}_{2})^{r}$ (hyperoctahedral
group). The detecting subalgebra $\e=\e_{\0}\oplus e_{\1}$ where $\e_{\1}$ is the span of
$\{E_{i,i+r}+E_{i+r,i}:\  i=1,2,\dots,r\}$ and $\e_{\0}$ is the span of $[\e_{1},\e_{\1}]$.
Constructions of detecting subalgebras for other classical Lie superalgebras are explicitly described in
\cite[Section 8]{BKN1}.

\subsection{The Witt algebra $W(n)$ and $S(n)$:}  Let $n\geq 2$, and 
$\Lambda^{\bullet}(V)$ be the exterior algebra of the vector space $V={\mathbb C}^n$. The Lie superalgebra
$W(n)$ is the set of all superderivations of $\Lambda^{\bullet}(V)$, 
and the superalgebra structure is provided via the supercommutator bracket. 
The Lie superalgebra $W(n)$ inherits a ${\mathbb Z}$-grading,
$${\mathfrak g}:=W(n)=W(n)_{-1}\oplus W(n)_0\oplus \dotsb \oplus W(n)_{n-1},$$
from $\Lambda^{\bullet}(V)$ by letting ${\mathfrak g}_{k}:=W(n)_k$ be the 
superderivations which increase the degree of
a homogeneous element by $k$.  Furthermore, the ${\mathbb Z}_2$-grading on 
$W(n)$ is obtained from the $\Z$-grading by taking
$W(n)_{\0}= \oplus_{k} W(n)_{2k}$ and $W(n)_{\1}= \oplus_{k} W(n)_{2k+1}$.

One can give an explicit basis for $W(n)$ in the following way. We have
$$W(n)\cong \Lambda(n)\otimes V^\ast.$$
We can fix an ordered basis $\{\xi _1, \dotsc , \xi _n\}$  for $V$. For each ordered subset $I=\{i_1, \dotsc , i_s\}$ of
$N=\{1, \dotsc , n\}$ with $i_1<i_2< \dotsb  <i_s$, let $\xi _I=\xi _{i_1}\xi _{i_2}\dotsb \xi _{i_s}$.
The set of all such $\xi_I$ forms a basis for $\Lambda(n)$.  If $1\leq i\leq n$ let $\partial _i$ be the
element of $W(n)$ given by partial differentiation: i.e., $\partial _i(\xi_j)=\delta _{ij}$. Write
$\xi_I\partial _i$ for $\xi_I\otimes \partial _i$. Then an
explicit basis for $W(n)$ is given by $\xi_{I}\partial_{i}$ where $I$ 
runs over all ordered subsets of $N$ and
$i=1,2,\dots,n$.  Observe that $W(n)_{k}$ is spanned by the basis 
elements $\xi_I\partial _i$ with $|I| =k+1$.
In particular, $\g_{0}$ is spanned by the elements $\xi_i\partial_j$, from which one
easily sees that $\g_0=\mathfrak{gl}(n)$.

Let ${\mathcal C}_{(\g,\g_{0})}$ be the category of $\g$-supermodules 
which are completely reducible over $\g_{0}$.
In \cite[Section 5.5]{BaKN}, a detecting subalgebra was constructed as follows. There
exists a generic point $x_{0}$ in ${\mathfrak g}_{-1}\oplus {\mathfrak g}_{1}$
for the action of $G_{0}=GL(n)$.
Set $H=\text{Stab}_{G_{0}}\ x_{0}=T_{n-1}$, an $(n-1)$-dimensional torus. 
The normalizer $N=N_{G_{0}}(H)$ is
$\Sigma_{n-1}\ltimes T$ where $\Sigma_{n-1}$ is the symmetric group of degree 
$n-1$ and $T$ is the torus of $G_{0}$ consisting
of diagonal matrices. Set $\f_{\1}=(\g_{-1}\oplus \g_{1})^{H}$. Then $\f_{1}$ is the span of the vectors
$\{\partial_{1},\xi_{1}\xi_{i}\partial_{i}:\ i=2,3,\dots,n\}$. 
Set $\f_{\0}=\text{Lie }(N)$. Then $\f=\f_{\0}\oplus \f_{\1}$
detects cohomology (as in the classical stable case), that is the restriction map:
$$
\text{res}:\text{H}^{\bullet}(\g,\g_{0},{\mathbb C})\rightarrow 
\text{H}^{\bullet}(\f,\f_{\0},{\mathbb C})^{N}
$$
is an isomorphism.

In \cite{BaKN} a support variety theory was developed for finite dimensional 
modules in ${\mathcal C}_{(\g,\g_{0})}$  (and
${\mathcal C}_{(\f,\f_{\0})}$), it was shown that for such a module $M$ there exists an injection of
algebraic varieties:
\begin{equation} \label{Cartaninjection}
{\mathcal V}_{(\f,\f_{\0})}(M)/\Sigma_{n-1}\hookrightarrow {\mathcal V}_{(\g,\g_{0})}(M).
\end{equation} 
In Theorem~\ref{T:isosupports2} we will prove that this is indeed an isomorphism of varieties. 

A similar development of cohomology and support varieties has been investigated by 
Bagci for the Lie superalgebra
of Cartan type $S(n)$. Detecting subalgebras analogous to $\f$ have been constructed. 
Details are omitted here and
left to the interested reader (see \cite{Ba}).

\section{Cohomological Embedding}

\subsection{The Stable Case: } The goal in this section is to prove that the 
relative $({\mathfrak g},{\mathfrak g}_{\0})$
cohomology for $M\in {\mathcal F}_{(\g,\g_{\0})}$ embeds in the relative cohomology for
$({\mathfrak f},{\mathfrak f}_{\0})$. Our first result uses a dimension shifting argument to reduce to 
looking at cohomology in degree one. 

\begin{prop}\label{P:reducecoho1} Let ${\mathfrak g}$ be a classical Lie 
superalgebra which is stable. Suppose that
$\operatorname{res}:\operatorname{H}^{1}(\g,\g_{\0},M)\rightarrow \operatorname{H}^{1}(\f,\f_{\0},M)$
is an injective map for every $M\in {\mathcal F}_{(\g,\g_{\0})}$. Then
$\operatorname{res}:\operatorname{H}^{n}(\g,\g_{\0},M)\rightarrow \operatorname{H}^{n}(\f,\f_{\0},M)$
is an injective map for all $n\geq 0$ and $M\in {\mathcal F}_{(\g,\g_{\0})}$.
\end{prop}

\begin{proof}  For $n=0$ the statement of the theorem
is clear because $\operatorname{H}^{0}(\g,\g_{\0},M)=M^{\mathfrak g}$ (fixed
points under ${\mathfrak g}$) and $\operatorname{H}^{0}(\f,\f_{\0},M)=M^{\mathfrak f}$. By assumption the
result holds for $n=1$. Now assume by induction the result holds for $n<t$, and consider the
short exact sequence
$$0\rightarrow M \rightarrow I \rightarrow \Omega^{-1}(M)\rightarrow 0$$
where $I$ is the injective hull of $M$ in ${\mathcal F}_{(\g,\g_{\0})}$. Note that
$I$ is also injective in ${\mathcal F}_{(\f,\f_{\0})}$.
Now applying the long exact sequence to the short exact sequence above and using these facts, we
have the following commutative diagram:
\begin{equation}
\CD
\text{H}^{t-1}(\g,\g_{\0},\Omega^{-1}(M)) @>\text{res}>> \text{H}^{t-1}(\f,\f_{\0},\Omega^{-1}(M))\\
@VVV @VVV\\
\text{H}^{t}(\g,\g_{\0},M) @>>\text{res}>  \text{H}^{t}(\f,\f_{\0},M)
\endCD
\end{equation}
where the vertical maps are isomorphisms. The top horizontal map is injective by induction.
Therefore, the bottom $\text{res}$ map is also injective.
\end{proof}

\subsection{} Let ${\mathfrak g}={\mathfrak g}_{-1}\oplus {\mathfrak g}_{0}\oplus {\mathfrak g}_{1}$
be a Type I Lie superalgebra. The detecting subalgebra ${\mathfrak f}$ has a triangular decomposition
which is compatible with the ${\mathbb Z}$-grading: ${\mathfrak f}={\mathfrak f}_{-1}\oplus
{\mathfrak f}_{0}\oplus {\mathfrak f}_{1}$ with $\f_{\pm 1}=\g_{\pm 1}^{H}$.  Note that
${\mathfrak f}^{\pm}={\mathfrak f}_{0}\oplus {\mathfrak f}_{\pm 1}$ and
${\mathfrak f}^{\pm}={\mathfrak f}\cap {\mathfrak g}^{\pm}$.

In order to analyze the question of embedding of cohomology of Type I Lie superalgebras we first
investigate the case $({\mathfrak g}^{\pm},{\mathfrak g}_{\0})$ and
$({\mathfrak f}^{\pm},{\mathfrak f}_{\0})$.

\begin{thm}\label{T:embedplusminus} Let ${\mathfrak g}$ be a classical Type I Lie superalgebra which is
stable. Then for all $M$ in ${\mathcal F}_{(\g,\g_{\0})}$ and
$n\neq 0$, the restriction map
$$\operatorname{H}^{n}({\mathfrak g}^{\pm},{\mathfrak g}_{\bar{0}},M)\rightarrow
\operatorname{H}^{n}({\mathfrak f}^{\pm},{\mathfrak f}_{\bar{0}},M)$$
is injective.
\end{thm}

\begin{proof} Without loss of generality we can consider the case 
${\mathfrak g}^{+}={\mathfrak g}_{0}\oplus
{\mathfrak g}_{1}$. Since $\g_{1}$ is an ideal of $\g^{+},$ one has the
Lyndon-Hochschild-Serre (LHS) spectral sequence
for the pair $(\g_{1}, \{0 \})$ in $(\g^{+}, \g_{0})$ (cf.\ proof of \cite[Theorem 3.3.1]{BKN3}):
$$
\widetilde{E}_{2}^{i,j}=\text{H}^{i}(\g_{0},\g_{0},\text{H}^{j}(\g_{1},M))\Rightarrow
\text{H}^{i+j}(\g^{+},\g_{0},M).
$$
This spectral sequence collapses because modules are completely reducible over ${\mathfrak g}_{0}$ and
yields:
\begin{equation}\label{E:giso}
\text{H}^{n}(\g^{+},\g_{0},M)\cong \text{H}^{n}(\g_{1},M)^{\g_{0}}
\end{equation}
for all $n\geq 0$. Similarly, for $n\geq 0$
\begin{equation}\label{E:fiso}
\text{H}^{n}(\f^{+},\f_{0},M)\cong \text{H}^{n}(\f_{1},M)^{\f_{0}}.
\end{equation}

Next observe that $\f_{1}$ is an abelian Lie superideal in $\g_{1}$. Conseqently, we have
another LHS spectral sequence:
$$
E_{2}^{i,j}=\text{H}^{i}(\g_{1}/\f_{1},\text{H}^{j}(\f_{1},M))\Rightarrow
\text{H}^{i+j}(\g_{1},M).
$$
This gives rise to an exact sequence (i.e., the first three terms of the standard
five term exact sequence):
\begin{equation}\label{E:5exactseq}
0\rightarrow \text{H}^{1}(\g_{1}/{\f_{1}},M^{\f_{1}}) \rightarrow \text{H}^{1}(\g_{1},M)
\rightarrow \text{H}^{1}(\f_{1},M)^{\g_{1}/\f_{1}}\rightarrow \dots
\end{equation}
Under the restriction map we have
$$\text{res}:\text{H}^{1}(\g_{1},M)^{\g_{0}}\rightarrow \text{H}^{1}(\f_{1},M)^{\f_{0}}.$$
In order to prove the theorem it suffices by Proposition~\ref{P:reducecoho1}, (~\ref{E:giso}), and
(~\ref{E:fiso}) to demonstrate that the restriction map above is injective.

The sequence (~\ref{E:5exactseq}) arises by looking at the following exact sequence
at the cochain level:
$$0\rightarrow (\g_{1}/\f_{1})^{*}\otimes M \xrightarrow{\alpha}
(\g_{1})^{*}\otimes M \xrightarrow{\beta} (\f_{1})^{*}\otimes M \rightarrow 0.$$
Consider $\text{Hom}_{G_{0}}(\g_{1},M)\cong [(\g_{1})^{*}\otimes M]^{G_{0}}$ as a subspace of
$(\g_{1})^{*}\otimes M$. The restriction of the original map
$\beta$ to this subspace $\beta:\text{Hom}_{G_{0}}(\g_{1},M)\rightarrow
[(\f_{1})^{*}\otimes M]^{\f_0}$ is given by $\beta(\psi)=\psi\mid_{\f_{1}}$. Now $\beta$ is
an injective map. This can be seen as follows.  If $\psi\mid_{\f_{1}}=0$ then $\psi(\f_{1})=0$. 
The fact that
$\psi$ is $G_{0}$-invariant shows that $\psi(G_{0}\cdot\f_{1})$. Finally using the
density of $G_{0}\cdot \f_{1}$ in $\g_{1}$ implies that $\psi=0$.

This means that $\text{Im }\alpha\cap [(\g_{1})^{*}\otimes M]^{G_{0}}=0$. The first
map in (~\ref{E:5exactseq}) is induced by $\alpha$, thus restricting
the second map to $\text{H}^{1}(\g_{1},M)^{G_{0}}$ yields an embedding
$\text{H}^{1}(\g_{1},M)^{G_{0}}\hookrightarrow \text{H}^{1}(\f_{1},M)^{\g_{1}/\f_{1}}$.
Note that the representative cocycles in $\text{H}^{1}(\g_{1},M)^{G_{0}}$ can be chosen to
be in $[(\g_{1})^{*}\otimes M]^{G_{0}}$ because the fixed point functor $(-)^{G_{0}}$ is exact.
\end{proof}


\subsection{} We now combine information from both sides of the triangular decomposition of
${\mathfrak g}$ and ${\mathfrak f}$ to prove that the relative cohomology for $(\g,\g_{\0})$
embeds in the relative cohomology for $(\f,\f_{\0})$.

\begin{thm}\label{T:stablecase} Let ${\mathfrak g}$ be a classical Type I Lie superalgebra which
is stable. Then for all $M\in {\mathcal F}_{({\mathfrak g},{\mathfrak g}_{\bar{0}})}$ and
$n\neq 0$ the restriction map
$$\operatorname{H}^{n}({\mathfrak g},{\mathfrak g}_{\bar{0}},M)\rightarrow
\operatorname{H}^{n}({\mathfrak f},{\mathfrak f}_{\bar{0}},M)$$
is injective.
\end{thm}

\begin{proof} For an explicit definition of relative Lie superalgebra cohomology via a 
cochain complex we refer the reader to
\cite[Section 2.3]{BKN1}. By Proposition~\ref{P:reducecoho1} it suffices to verify the case $n=1$.
Consider cochain differentials whose images are respectively used to define 
$\text{H}^{1}(\f^{\pm},\f_{\0},M)$,
$\text{H}^{1}(\g^{\pm},\g_{\0},M)$, $\text{H}^{1}(\f,\f_{\0},M)$, and  $\text{H}^{1}(\g,\g_{\0},M)$:
\begin{equation*}
d_{\f^{\pm}}:M^{\f_{\0}}\rightarrow [(\f_{\pm 1})^{*}\otimes M]^{\f_{\0}},
\end{equation*}
\begin{equation*}
d_{\g^{\pm}}:M^{\g_{\0}}\rightarrow [(\g_{\pm 1})^{*}\otimes M]^{\g_{\0}},
\end{equation*}
\begin{equation*}
d_{\f}:M^{\f_{\0}}\rightarrow [(\f_{\1})^{*}\otimes M]^{\f_{\0}},
\end{equation*}
\begin{equation*}
d_{\g}:M^{\g_{\0}}\rightarrow [(\g_{\1})^{*}\otimes M]^{\g_{\0}}.
\end{equation*}
In Theorem~\ref{T:embedplusminus}, we proved that the restriction map
embeds $\text{H}^{1}(\g^{\pm},\g_{\0},M)$ into $\text{H}^{1}(\f^{\pm},\f_{\0},M)$.
At the cochain level, we have the following commutative diagram:
\begin{equation}
\CD
M^{\g_{\0}} @>>> M^{\f_{\0}}\\
@Vd_{\g^{\pm}}VV @Vd_{\f^{\pm}}VV\\
[(\g_{\pm 1})^{*}\otimes M]^{\g_{\0}} @>>\sigma_{\pm}>  [(\f_{\pm})^{*}\otimes M]^{\f_{\0}}
\endCD
\end{equation}
where $\sigma_{\pm}$ is the map obtained by restriction of functions.
The embedding of $\text{H}^{1}(\g^{\pm},\g_{\0},M)$ into $\text{H}^{1}(\f^{\pm},\f_{\0},M)$
implies that
\begin{equation}
\sigma_{\pm}^{-1}(\text{Im }d_{\f^{\pm}})=\text{Im }d_{\g^{\pm}}
\end{equation}
or
\begin{equation}
\text{Im }d_{\f^{\pm}}=\sigma_{\pm}(\text{Im }d_{\g^{\pm}}).
\end{equation}
Next note that $\text{Im }d_{\f^{\pm}}\cong M^{\f_{\0}}/\text{Ker }d_{\f^{\pm}}$ and
$\text{Im }d_{\g^{\pm}}\cong M^{\g_{\0}}/\text{Ker }d_{\g^{\pm}}$. Since $G_{\0}\cdot {\mathfrak f}_{\1}$ is
dense in ${\mathfrak g}_{\1}$ the map obtained by restriction of functions
$$\sigma:[(\g_{\1})^{*}\otimes M]^{\g_{\0}} \hookrightarrow  [(\f_{\1})^{*}\otimes M]^{\f_{\0}}$$
is injective.  In order to prove that the induced map in cohomology 
from $\text{H}^{1}(\g,\g_{\0},M)\rightarrow
\text{H}^{1}(\f,\f_{\0},M)$ is injective, it suffices (using reasoning similar to that above) to
prove that $\sigma(\text{Im }d_{\g})=\text{Im }d_{\f}$.

From the definition of the differential note that
\begin{equation}
K_{\f}:=\text{Ker }d_{\f}=\text{Ker }d_{\f^{-}}\cap \text{Ker }d_{\f^{+}},
\end{equation}
\begin{equation}
K_{\g}:=\text{Ker }d_{\g}=\text{Ker }d_{\g^{-}}\cap \text{Ker }d_{\g^{+}}
\end{equation}
because $\text{Ker }d_{ \f^\pm}=M^{f^\pm}$ and $\text{Ker }d_{\g^\pm}=M^{ g^\pm}$.
We have the following commutative diagram
\begin{equation}
\CD
0 @>>> \text{Ker }d_{\g^{+}}/K_{\g} @>>> M^{\g_{\0}}/K_{\g} @>>> M^{\g_{\0}}/\text{Ker }d_{\g^{+}} @>>> 0 \\
@.       @V\text{$\sigma$}VV                      @V\text{$\sigma$}VV                 @V\text{${\sigma}_{+}$}VV    @.  \\
0 @>>> \text{Ker }d_{\f^{+}}/K_{\f} @>>> M^{\f_{\0}}/K_{\f} @>>> M^{\f_{\0}}/\text{Ker }d_{\f^{+}} @>>> 0
\endCD
\end{equation}
From our analysis above the map $\sigma_{+}$ is an isomorphism. Therefore, in order to prove the
theorem it suffices to show that
$$\sigma: \text{Ker }d_{\g^{+}}/K_{\g}\rightarrow \text{Ker }d_{\f^{+}}/K_{\f}$$
is an isomorphism because the five lemma would imply that the map in the middle is an isomorphism.
This would show that $\sigma(\text{Im }d_{\g})=\text{Im }d_{\f}$.

By the second isomorphism theorem we have the following isomorphisms:
\begin{equation}
\text{Ker }d_{\f^{+}}/K_{\f}\cong (\text{Ker }d_{\f^{+}}+\text{Ker }d_{\f^{-}})/\text{Ker d}_{\f^{-}},
\end{equation}
\begin{equation}
\text{Ker }d_{\g^{+}}/K_{\g}\cong (\text{Ker }d_{\g^{+}}+\text{Ker }d_{\g^{-}})/\text{Ker d}_{\g^{-}}.
\end{equation}
Now observe that from the relationship between
$(\g^{-},\g_{\0})$ and $(\f^{-},\f_{\0})$, we have a commutative diagram
\begin{equation}
\CD
(\text{Ker }d_{\g^{+}}+\text{Ker }d_{\g^{-}})/\text{Ker }d_{\g^{-}} @>>> M^{\g_{\0}}/\text{Ker }d_{\g^{-}}\\
@V\text{$\sigma_{-}$}VV      @V\text{$\sigma_{-}$}VV \\
(\text{Ker }d_{\f^{+}}+\text{Ker }d_{\f^{-}})/\text{Ker }d_{\f^{-}} @>>> M^{\f_{\0}}/\text{Ker }d_{\f^{-}}
\endCD
\end{equation}
The horizontal maps are embeddings and the rightmost vertical map is an isomorphism
(using the fact that $\text{Im }d_{\f^{-}}=\sigma_{-}(\text{Im }d_{\g^{-}})$). Therefore, the map
$$\sigma_{-}:(\text{Ker }d_{\g^{+}}+\text{Ker }d_{\g^{-}})/\text{Ker  }d_{\g^{-}}\hookrightarrow
(\text{Ker }d_{\f^{+}}+\text{Ker }d_{\f^{-}})/\text{Ker  }d_{\f^{-}}$$
is injective. In order to finish the proof we need to show that $\sigma_{-}$ is surjective.

Suppose that 
$y+\text{Ker }d_{\f^{-}}\in (\text{Ker }d_{\f^{+}}+\text{Ker }d_{\f^{-}})/\text{Ker d}_{\f^{-}}$ with
$y=y_{-}+y_{+}$ where $y_{\pm}\in \text{Ker }d_{\f^{\pm}}$. From the isomorphism given by $\sigma_{-}$ above
we have $g(y_{\pm}+\text{Ker }d_{\f_{-}})=y_{\pm}+\text{Ker }d_{\f^{-}}$ for all $g\in G_{\0}$.
Moreover, $d_{\f^{\pm}}(y_{\pm})=0$ so in particular $\f_{1}.y_{+}=0$ and $\f_{-1}.y_{-}=0$. Now
$0=g.(f.y_{\pm})=(g.f).(g^{-1}.y_{\pm})=(g.f).y_{\pm}$ for all $g\in G_{\0}$ and $f\in \f_{\pm 1}$.
Since $G_{\0}\cdot \f_{\pm 1}$ is dense in $\g_{\pm 1}$ it follows that $x.y_{\pm}=0$ for all
$x\in \g_{\pm 1}$, thus $y_{\pm}\in \text{Ker }d_{\g^{\pm}}$.
\end{proof}

\subsection{} Let ${\mathfrak h}$ be a classical Lie subsuperalgebra of ${\mathfrak g}$ with the
property that
\begin{equation}\label{E:propertyinj}
\text{res}:\text{H}^{n}(\g,\g_{\0},M)\hookrightarrow \text{H}^{n}(\h,\h_{\0},M)
\end{equation}
is an injective map for all $n\geq 0$ and $M$ in ${\mathcal F}_{(\g,\g_{\0})}$.

Let $M$ be a module in ${\mathcal F}_{(\g,\g_{\0})}$ which is projective when
considered as a module in ${\mathcal F}_{({\mathfrak h},{\mathfrak h}_{\0})}$. By
(~\ref{E:propertyinj}), we have
\begin{equation*}
\text{res}:\text{H}^{1}(\g,\g_{\0},M\otimes S^*)\hookrightarrow \text{H}^{1}(\h,\h_{\0},M\otimes S^*)
\end{equation*}
for all simple modules $S$ in ${\mathcal F}_{({\mathfrak g},{\mathfrak g}_{\0})}$. By the
projectivity of $M$ in ${\mathcal F}_{({\mathfrak h},{\mathfrak h}_{\0})}$, we further have
$\text{H}^{1}(\h,\h_{\0},M\otimes S^*)=0$. Consequently,
$$0=\text{H}^{1}(\g,\g_{\0},M\otimes S^*)\cong \text{Ext}^{1}_{{\mathcal F}_{(\g,\g_{\0})}}(S,M)$$
for all simple modules $S$ in ${\mathcal F}_{({\mathfrak g},{\mathfrak g}_{\0})}$, whence
$M$ is projective in ${\mathcal F}_{({\mathfrak g},{\mathfrak g}_{\0})}$. For Type I classical
Lie superalgebras when ${\mathfrak h}={\mathfrak f}$ this fact about projectivity
was earlier deduced using geometric methods involving support varieties
(cf. \cite[Theorem 3.5.1, Theorem 3.7.1]{BKN3}). Theorem~\ref{T:stablecase} can be viewed as
a strong generalization of this projectivity result between modules in ${\mathcal F}_{(\g,\g_{\0})}$ and
${\mathcal F}_{(\f,\f_{\0})}$

\subsection{The Polar Case: } In this section we will consider the example
when ${\mathfrak g}=\mathfrak{gl}(1|1)$ and
demonstrate that the analogue Theorem~\ref{T:stablecase} 
does not hold when the detecting subalgebra $\f$ is
replaced by $\e$. The simple modules in the principal block 
of ${\mathcal F}_{(\g,\g_{\0})}$ are
labelled by $L(\lambda)$ where $\lambda$ denotes the highest weight 
$(\lambda|-\lambda)$. Consider the two-dimensional
dual Kac module $K^{-}(\lambda)$. The module $K^{-}(\lambda)$ remains indecomposable when restricted to
$\e$, and the only 2-dimensional indecomposable $\e$-modules are the projective indecomposable
modules in the category ${\mathcal F}_{(\e,\e_{\0})}$. Consequently, $K^{-}(\lambda)$ is a projective
module in ${\mathcal F}_{(\e,e_{\0})}$, thus $\text{H}^{n}(\e,\e_{\0},K^{-}(\lambda))=0$ for $n>0$.
It is well known that the Kac modules in the principal block are not
projective in ${\mathcal F}_{(\g,\g_{\0})}$,
thus by the argument in Section 3.4, the property (~\ref{E:propertyinj}) 
cannot hold for ${\mathfrak h}=
{\mathfrak e}$.

We also see that the restriction map is not injective by direct computation. By Frobenius reciprocity
\begin{equation}
\text{H}^{n}(\g,\g_{\0},K^{-}(\lambda))\cong \text{H}^{n}(\g^{-},\g_{\0},\lambda)
\cong [\text{H}^{n}(\g_{-1},{\mathbb C})\otimes \lambda]^{G_{\0}}\cong 
[S^{n}((\g_{-1})^*)\otimes \lambda]^{G_{\0}}.
\end{equation}
In this instance $G_{\0}$ a torus (specifically, the set of invertible diagonal matrices). 
The subspace $\g_{-1}$ is one-dimensional
and is spanned by a weight vector having weight $(-1|1)$. Consequently,
\begin{equation}
\text{H}^{n}(\g,\g_{\0},K^{-}(\lambda))\cong \begin{cases} \C  & \text{$\lambda=-n$}\\
                                                           0 & \text{otherwise}.
\end{cases}
\end{equation}
In particular, we have $\text{H}^{1}(\g,\g_{\0},K^{-}(-1|1))\cong \C$, and
$\text{H}^{1}(\e,\e_{\0},K^{-}(-1,1))=0$. Therefore,
$$\text{res}:\text{H}^{1}(\g,\g_{\0},K^{-}(-1|1))\rightarrow \text{H}^{1}(\e,\e_{\0},K^{-}(-1|1))$$
is not an injective map. However,
$$\text{res}:\text{H}^{n}(\g,\g_{\0},K^{-}(-1|1))\rightarrow \text{H}^{n}(\e,\e_{\0},K^{-}(-1|1))$$
is an injective map for $n\geq 2$.

It was shown in \cite{BKN1} that when $\dim \e_{\1}=1$ the
restriction map is injective for $n$ sufficiently large. An interesting problem would be
to determine whether this occurs for arbitrary $\e$.

\subsection{Types $W(n)$ and $S(n)$: } The techniques used to prove Theorem~\ref{T:embedplusminus} 
and ~\ref{T:stablecase} can
be used with the triangular decomposition given by the ${\mathbb Z}$-grading for $W(n)$ and $S(n)$ to 
prove the following detection theorem.
Let $\f$ be as in Section 2.5.

\begin{thm}\label{T:WSembedplusminus} Let ${\mathfrak g}=W(n)$ or $S(n)$. 
Then for all $M$ in ${\mathcal C}_{(\g,\g_{0})}$ and
$n\neq 0$. The restriction map
$$\operatorname{H}^{n}({\mathfrak g}^{\pm},{\mathfrak g}_{\bar{0}},M)\rightarrow
\operatorname{H}^{n}({\mathfrak f}^{\pm},{\mathfrak f}_{\bar{0}},M).$$
is injective.
\end{thm}

\section{Support Varieties}

\subsection{} We first recall the definition of the support variety
of a finite dimensional $\g$-supermodule $M$ (cf. \cite[Section 6.1]{BKN1}). 
Let $\mathfrak{g}$ be a classical Lie superalgebra,
$R:=\operatorname{H}^{\bullet}(\mathfrak{g}, \mathfrak{g}_{\0};{\mathbb C})$, and $M_{1}$, $M_{2}$
be in ${\mathcal F}:={\mathcal F}_{(\g,\g_{\0})}$. According to \cite[Theorem 2.5.3]{BKN1},
$\Ext_{\mathcal{F}}^{\bullet}(M_{1},M_{2})$ is a finitely generated $R$-module.
Set $J_{({\mathfrak g},{\mathfrak g}_{\0})}(M_{1},M_{2})=
\operatorname{Ann}_{R}(\Ext_{\mathcal{F}}^{\bullet}(M_{1},M_{2}))$
(i.e., the annihilator ideal of this module).  The \emph{relative support variety of the pair $(M,N)$} is

\begin{equation}
\mathcal{V}_{(\mathfrak{g},\mathfrak{g}_{\0})}(M,N)=
\operatorname{MaxSpec}(R/J_{({\mathfrak g},{\mathfrak g}_{\0})}(M,N))
\end{equation}

In the case when $M=M_{1}=M_{2}$, set $J_{(\g,\g_{\0})}(M)=J_{(\g,\g_{\0})}(M,M)$, and
$$\mathcal{V}_{(\g,\g_{\0})}(M):=\mathcal{V}_{(\g,\g_{\0})}(M,M).$$
The variety $\mathcal{V}_{(\mathfrak{g},\mathfrak{g}_{\0})}(M)$ is called the \emph{support variety} of $M$.
In this situation, $J_{(\g,\g_{\0})}(M)=\text{Ann}_{R} \ \text{Id}$ where $\text{Id}$ is the identity
morphism in $\text{Ext}^{0}_{\mathcal F}(M,M)$.

\subsection{} We will now compare support varieties for the classical Lie superalgebras ${\mathfrak g}$,
${\mathfrak f}$, and ${\mathfrak e}$. Assume that $\g$ is both stable and polar. Without
the assumption that $\g$ is polar, the statements concerning cohomology and
support varieties for ${\mathfrak g}$ and ${\mathfrak f}$ remain true.

First there are natural maps of rings given by restriction:
$$
\text{res}:\operatorname{H}^{\bullet}(\g,\g_{\0}; \C) \rightarrow \operatorname{H}^{\bullet}(\f,\f_{\0}; \C)
\rightarrow \operatorname{H}^{\bullet}(\e,\e_{\0},\C).
$$
which induce isomorphisms
$$
\text{res}:\operatorname{H}^{\bullet}(\g,\g_{\0}; \C) \rightarrow 
\operatorname{H}^{\bullet}(\f,\f_{\0}; \C)^{N/N_{0}}
\rightarrow \operatorname{H}^{\bullet}(\e,\e_{\0},\C)^{W}.
$$
The map on cohomology above induces morphisms of varieties:
$$
\text{res}^{*}:\mathcal{V}_{(\e,\e_{\0})}(\C)\rightarrow
\mathcal{V}_{(\f,\f_{\0})}(\C)\rightarrow \mathcal{V}_{(\g,\g_{\0})}(\C)
$$
and isomorphisms (by passing to quotient spaces)
$$
\text{res}^{*}:\mathcal{V}_{(\e,\e_{\0})}(\C)/W\rightarrow
\mathcal{V}_{(\f,\f_{\0})}(\C)/(N/N_{0})\rightarrow \mathcal{V}_{(\g,\g_{\0})}(\C).
$$

Let $M$ be a finite dimensional $\g$-module. Then $\res^{*}$ induces maps
between support varieties:
$$\mathcal{V}_{(\e,\e_{\0})}(M) \rightarrow \mathcal{V}_{(\f,\f_{\0})}(M) \rightarrow
\mathcal{V}_{(\g,\g_{\0})}(M).
$$
Since $M$ is ${\mathfrak g}_{\0}$-module, the first two varieties are stable under the action of
$W$ and $N/N^{0},$ respectively. Consequently, we obtain the following induced maps of
varieties:
$$\mathcal{V}_{(\e,\e_{\0})}(M)/W \hookrightarrow \mathcal{V}_{(\f,\f_{\0})}(M)/(N/N_{0})
\hookrightarrow \mathcal{V}_{(\g,\g_{\0})}(M).
$$
These maps are embeddings because if $x\in R$ annihilates the identity in 
$\text{H}^{0}(\g,\g_{\0},M^{*}\otimes M)$ then
it must annihilate the identity elements in $\text{H}^{0}(\f,\f_{\0},M^{*}\otimes M)$ and
$\text{H}^{0}(\e,\e_{\0},M^{*}\otimes M)$.

\subsection{The Intermediate Subalgebra $\overline{\f}$: } We next define an
intermediate subalgebra between $\e$ and $\f$ which will be useful for our purposes. Let
$\overline{\f}$ be defined as follows. Given $\f$, let $\overline{\f}_{\1}:=\f_{\1}$ and
$\overline{\f}_{\0}=\text{Lie}(H)$. Set $\overline{\f}=\overline{\f}_{\0}\oplus {\f}_{\1}$.
From the proof of \cite[Theorem 4.1]{BKN1}, we know that $\overline{\f}$ is a Lie subsuperalgebra of
$\f$ and contains $\e$ (in the case that we have a polar action). Moreover,
$[\overline{\f}_{\0},\overline{\f}_{\1}]=0$. This implies that we have a rank variety description for
${\mathcal V}_{(\overline{\f},\overline{\f}_{\0})}(M)$ when 
$M\in {\mathcal F}_{(\overline{\f},\overline{\f}_{\0})}$.

The \emph{rank variety} of $M$ is
$$
{\mathcal V}_{\overline{\f}}^{\text{rank}}(M)=\left\{ x \in\overline{\f}_{\1}\:\vert\:  M
\text{ is not projective as $U(\langle x \rangle)$-module} \right\} \cup \{0 \}.
$$
It was shown in \cite[Theorem 6.3.2]{BKN1} that there is an isomorphism
$${\mathcal V}_{(\overline{\f},\overline{\f}_{\0})}(M)\cong {\mathcal V}_{\overline{\f}}^{\text{rank}}(M).$$

\subsection{Comparing support varieties over $\f$ and $\overline{\f}$: } In this section we
compare support varieties for modules over $\overline{\f}$ and $\f$. Let
$M \in \mathcal{F}_{({\f},{\f}_{\0})}$. There is a map of varieties
induced by the restriction map in cohomology:
\begin{equation}
\text{res}^{*}:{\mathcal V} _{(\overline{\f},\overline{\f}_{\0})}(M) \rightarrow 
{\mathcal V} _{(\f,\f_{\0})}(M).
\end{equation}
Using the fact that
$$\text{H}^{\bullet}(\f,\f_{\0},{\mathbb C})\cong \text{H}^{\bullet}(\overline{\f},
\overline{\f}_{\0},{\mathbb C})^{N_{0}/H},
$$
it is clear that the above map is the restriction of the orbit map:
\begin{equation}\label{E:trivialresequality}
{\mathcal V} _{(\overline{\f},\overline{\f}_{\0})}({\mathbb C})
\longrightarrow{\mathcal V} _{(\overline{\f},\overline{\f}_{\0})}({\mathbb C})/(N_{0}/H) 
\overset{\sim}{\longrightarrow}
{\mathcal V}_{({\f},{\f}_{\0})}({\mathbb C}).
\end{equation}

The following theorem demonstrates that support varieties in this context
are natural with respect to taking quotients.

\begin{thm}\label{T:supportquotient}
Let $M$ be a finite dimensional object in $\mathcal{F}_{(\f,\f_{\0})}$. Then
\begin{equation}\label{E:fres}
{\mathcal V}_{(\overline{\f},\overline{\f}_{\0})}(M)/(N_{0}/H)\cong \operatorname{res}^{*}
({\mathcal V}_{(\overline{\f},\overline{\f}_{\0})}(M))={\mathcal V}_{(\f,\f_{\0})}(M).
\end{equation}
\end{thm}

\begin{proof} The proof, which we include
for the convenience of the reader, will follow the same lines as \cite[Theorem 6.4.1]{BaKN}.
Observe that the first isomorphism holds by \eqref{E:trivialresequality}. To show that the
second isomorphism holds we need to prove that $\text{res}^{*}$ is a surjective map.

The group $N_{0}/H$ is reductive, and therefore has completely reducible module category.
Let $X_+$ be a parametrizing set for the finite-dimensional
simple $N_{0}/H$-modules, and for $\lambda\in X_+$, let $S_{\lambda}$ be the corresponding simple module.
Let $Q$ be a module in $\mathcal{F}_{(\f,\f_{\0})}$. Then $N_{0}/H$ acts on the cohomology
$\text{H}^{\bullet}(\overline{\f},\overline{\f}_{\0},Q)$, and by complete reducibility,
\begin{eqnarray*}
\text{H}^{\bullet}(\overline{\f},\overline{\f}_{\0},Q) &\cong &
\text{H}^{\bullet}(\overline{\f},\overline{\f}_{\0},Q)^{N_{0}/H} \oplus
\bigoplus_{\lambda \in X_{+}:\ \lambda \neq 0} \text{Hom}_{N_{0}/H}(S_{\lambda},
\text{H}^{\bullet}(\overline{\f},\overline{\f}_{\0},Q))\otimes S_{\lambda}  \\
&\cong & \text{H}^{\bullet}(\f,\f_{\0}; Q) \oplus
\bigoplus_{\lambda \in X_{+}:\ \lambda \neq 0} \text{Hom}_{N_{0}/H}(S_{\lambda},
\text{H}^{\bullet}(\overline{\f},\overline{\f}_{\0},Q))\otimes S_{\lambda}.
\end{eqnarray*}
From this isomorphism, one sees that for all $Q$ in $\mathcal{F}_{(\f,\f_{\0})}$
\begin{equation}\label{E:embedgeneralcohom}
\operatorname{res} : \text{H}^{\bullet}(\f,\f_{\0}; Q) \xrightarrow{\cong}
\text{H}^{\bullet}(\overline{\f},\overline{\f}_{\0}; Q)^{N_{0}/H} \subseteq
\text{H}^{\bullet}(\overline{\f},\overline{\f}_{\0}; Q).
\end{equation}
Let $\text{id}_{\f,M}$ (resp. $\text{id}_{\overline{\f},M}$) denote the identity element in
$\text{H}^{\bullet}(\f,\f_{\0}, M^*\otimes M)$ (resp. $\text{H}^{\bullet}(\overline{\f},
\overline{\f}_{\0}, M^*\otimes M)$). The ideal
$\operatorname{res}^{-1}( J_{(\overline{\f}, \overline{\f}_{\0})}(M))$ defines the
variety $\text{res}^{*}({\mathcal V}_{(\overline{\f},\overline{\f}_{\0})}(M))$.
We need to prove that
\begin{equation}
\operatorname{res}^{-1}(J_{(\overline{\f},\overline{\f}_{\0})}(M))=J_{(\f,\f_{\0})}(M).
\end{equation}

If $x \in J_{(\f,\f_{\0})}(M)$, then
\begin{equation*}
0=\operatorname{res}(x.\text{id}_{\f, M})=
\operatorname{res}(x).\operatorname{res}(\text{id}_{\f ,M})=
\operatorname{res}(x).\text{id}_{\overline{\f},M}.
\end{equation*}
Therefore, $\operatorname{res}(x) \in J_{(\overline{\f },\overline{\f}_{\0})}(M)$, and so
$x \in \operatorname{res}^{-1}(J_{(\overline{\f },\overline{\f}_{\0})}(M)).$
Conversely, if $x \in \operatorname{res}^{-1}(J_{(\overline{\f},\overline{\f}_{\0})}(M))$, then
\begin{equation*}
0=\operatorname{res}(x).\text{id}_{\overline{\f},M}=\operatorname{res}(x).
\operatorname{res}(\text{id}_{\f,M})
=\operatorname{res}(x.\text{id}_{\f,M}).
\end{equation*}
Since the restriction map is injective, $0=x.\text{id}_{\f ,M}$ and so $x \in J_{(\f,\f_{\0})}(M)$.
\end{proof}

Using the identification of ${\mathcal V}_{(\overline{\f},\overline{\f}_{\0})}(M)$ as a rank variety,
Theorem~\ref{T:supportquotient} provides a concrete realization of
${\mathcal V}_{({\f},\f_{\0})}(M)$. Since the tensor product theorem holds for rank varieties
it follows that it also holds for support varieties in this context. The proof
of \cite[Theorem 6.5.1]{BaKN} adapted to our setting yields the following result.

\begin{cor}\label{T:fvarietyprops}Let $M$, $N$ be modules in
${\mathcal F}_{(\f,\f_{\0})}$. Then
\begin{itemize}
\item [(a)] ${\mathcal V}_{(\f,\f_{\0})}(M)\cong 
{\mathcal V}^{\operatorname{rank}}_{\overline{\f}}(M)/(N_{0}/H)$,
\item[(b)] ${\mathcal V}_{(\f,\f_{\0})}(M\otimes N)={\mathcal V}_{(\f,\f_{\0} )}(M)\cap
{\mathcal V}_{(\f,\f_{\0} )}(N)$.
\end{itemize}
\end{cor}
\subsection{Comparing support varieties over $\f$ and $\e$: } In this section we compare the
varieties for the detecting subalgebras $\f$ and $\e$. In general $\f$ does not have a simple
rank variety description. This necessitates the use of the auxiliary algebra $\overline{\f}$
to make the transition between $\f$ and $\e$.

\begin{thm}\label{T:isosupportsfe}  Let $\g$ be a classical Lie superalgebra which is
stable and polar. If $M\in {\mathcal F}_{(\f,\f_{\0})}$ then
we have the following isomorphism of varieties:
$$\operatorname{res}^{*}:{\mathcal V}_{(\e,\e_{\0})}(M)/W\rightarrow 
{\mathcal V}_{(\f,\f_{\0})}(M)/(N/N_{0}).$$
\end{thm}

\begin{proof} Let $M\in {\mathcal F}_{(\f,\f_{\0})}$. We have the following commutative diagram of varieties:
\begin{equation*}
\CD
{\mathcal V}_{(\e,\e_{\0})}(M) @>\text{res}^{*}>> {\mathcal V}_{(\overline{\f},\overline{\f}_{\0})}(M)
@>\text{res}^{*}>> {\mathcal V}_{(\f,\f_{\0})}(M) \\
@VVV      @VVV @VVV \\
{\mathcal V}^{\text{rank}}_{\e}(M) @>>> {\mathcal V}^{\text{rank}}_{\overline{\f}}(M)
@>{\beta} >> {\mathcal V}^{\text{rank}}_{\overline{\f}}(M)/(N_{0}/H)\\
@VVV      @VVV @VVV \\
{\mathcal V}^{\text{rank}}_{\e}({\mathbb C}) @>\alpha>> {\mathcal V}^{\text{rank}}_{\overline{\f}}({\mathbb C})
@>{\beta} >> {\mathcal V}^{\text{rank}}_{\overline{\f}}({\mathbb C})/(N_{0}/H)
\endCD
\end{equation*}
It suffices to show that the composition of the top (horizontal) maps
$$\operatorname{res}^{*}:{\mathcal V}_{(\e,\e_{\0})}(M)\rightarrow {\mathcal V}_{(\f,\f_{\0})}(M)$$
is surjective.

Let $\sigma$ denote the middle (horizontal) composition of maps from
${\mathcal V}^{\text{rank}}_{\e}(M)$ to ${\mathcal V}^{\text{rank}}_{\overline{\f}}(M)/(N_{0}/H)$.
Since the first row of vertical maps are all isomorphisms,
it suffices to prove that $\sigma$ is surjective.
Take $y\in {\mathcal V}^{\text{rank}}_{\overline{\f}}(M)/(N_{0}/H)$.
There exists $x\in {\mathcal V}^{\text{rank}}_{\overline{\f}}(M) $ with $\beta(x)=y$,
and by naturality of rank varieties, we may consider
$x$ to be an element of ${\mathcal V}^{\text{rank}}_{\overline\f}({\mathbb C})$.
Since the composition $\beta\circ\alpha$ of the two lowest horizontal arrows is an isomorphism, there is
an element $z\in {\mathcal V}^{\text{rank}}_{\e}({\mathbb C})$ such that $\beta(x)=\beta(z)$,
where, since $\alpha$ is an inclusion, we identify $z$ as an element of
${\mathcal V}^{\text{rank}}_{\overline\f}({\mathbb C})$. But then $z=gx$ for some $g\in N_0/H$,
and it follows that
$z\in {\mathcal V}^{\text{rank}}_{\e}(M)$, and clearly $\sigma(z)=y$.
\end{proof}

\subsection{} In the case of ${\mathfrak g}=W(n)$ or $S(n)$ there is an auxilliary subalgebra
$\widetilde{\f}$ which is analogous to $\overline{\f}$. In this setting we have the following result
(cf. \cite[Theorem 6.4.1, (6.3.4)]{BaKN}).

\begin{thm}\label{T:supportquotient2} Let ${\mathfrak g}=W(n)$ or $S(n)$ and
let $M$ be a finite dimensional object in $\mathcal{C}_{(\f,\f_{0})}$. Then there exists a
torus $T$ such that
\begin{itemize}
\item[(a)] ${\mathcal V}_{(\widetilde{\f},\widetilde{\f}_{0})}(M)/T\cong \operatorname{res}^{*}
({\mathcal V}_{(\overline{\f},\overline{\f}_{0})}(M))={\mathcal V}_{(\f,\f_{0})}(M)$.
\item[(b)] ${\mathcal V}_{(\f,\f_{0})}(M)/T\cong {\mathcal V}^{\operatorname{rank}}_{\widetilde{\f}}(M)$.
\end{itemize}
\end{thm}


\section{Applications}

\subsection{Isomorphism Theorems} In \cite{BKN1,BKN2} there was convincing theoretical and
computational evidence of direct relationships between
the support varieties for a classical Lie superalgebra ${\mathfrak g}$
and those of its detecting subalgebras. The cohomological embedding
theorem provided in Section 3 enables us to provide such a relation, which we now proceed
to do.

\begin{thm}\label{T:isosupports}  Let $\g$ be a classical Type I Lie superalgebra
with ${\mathfrak f}$ etc. as above, and let $M$ be in the module category ${\mathcal F_{(\g,\g_{\0})}}$.
\begin{itemize}
\item[(a)] If $\g$ is stable then the map on support varieties
$$\operatorname{res}^{*}:{\mathcal V}_{(\f,\f_{\0})}(M)/(N/N_{0})
\rightarrow {\mathcal V}_{(\g,\g_{\0})}(M)$$ is an isomorphism.
\item[(b)] If $\g$ is stable and polar then the maps on support varieties
$$\operatorname{res}^{*}:
{\mathcal V}_{(\e,\e_{\0})}(M)/W({\e})\rightarrow {\mathcal V}_{(\f,\f_{\0})}(M)/(N/N_{0})
\rightarrow {\mathcal V}_{(\g,\g_{\0})}(M)$$
are isomorphisms, where $W=W(\e)$ is the pseudoreflection group of (2.2.3).
\end{itemize}
\end{thm}

\begin{proof} (a) We have seen that $\text{res}^{*}$ is an embedding. Therefore,
it suffices to show that this map is surjective. Observe that by Theorem~\ref{T:stablecase},
the restriction map
$$\text{H}^{\bullet}(\g,\g_{\0},M^{*}\otimes M)\hookrightarrow
\text{H}^{\bullet}(\f,\f_{\0},M^{*}\otimes M)$$
is an injection. Therefore, if $x\in R$ annihilates $\text{H}^{\bullet}(\f,\f_{\0},M^{*}\otimes M)$
then it annihilates $\text{H}^{\bullet}(\g,\g_{\0},M^{*}\otimes M)$, and it follows that
$$\text{Ann}_{R}\ \text{H}^{\bullet}(\g,\g_{\0},M^{*}\otimes M)=\text{Ann}_{R}\ 
\text{H}^{\bullet}(\f,\f_{\0},M^{*}\otimes M).$$

The statement (b) follows using part (a) and Theorem~\ref{T:isosupportsfe}.
\end{proof}

For ${\mathfrak g}=W(n)$, $S(n)$ we have the following result, which verifies the conjecture
made at the end of \cite[Section 6.2]{BaKN}.

\begin{thm}\label{T:isosupports2}  Let $\g$ be $W(n)$ or $S(n)$ and let $M$ be
a finite dimensional module in
${\mathcal C}_{(\g,\g_{0})}$. The induced map of support varieties
$$\operatorname{res}^{*}:{\mathcal V}_{(\f,\f_{\0})}(M)/W
\rightarrow {\mathcal V}_{(\g,\g_{\0})}(M)$$
is an isomorphism where $W=\Sigma_{n-1}$ (resp. $\Sigma_{n-2}$) for $W(n)$ (resp. $S(n)$).
\end{thm}

\begin{proof} This follows by using the analogue of Theorem~\ref{T:stablecase} for $\g=W(n)$ or
$S(n)$ and applying the argument given above in the proof of Theorem~\ref{T:isosupports}(a).
\end{proof}

\subsection{Realizability and the Tensor Product Theorem: } Theorem~\ref{T:isosupports}
allows us to provide a concrete realization for the variety ${\mathcal V}_{(\g,\g_{\0})}(M)$
when $M\in {\mathcal F}_{(\g,\g_{\0})}$. The
next theorem provides a rank variety description of ${\mathcal V}_{(\g,\g_{\0})}(M)$ when
$M$ is in ${\mathcal F}_{(\g,\g_{\0})}$.

\begin{thm} \label{T:supportprop} Let ${\mathfrak g}$ be a Type I classical Lie superalgebra 
which is both stable and
polar, and let $M_{1}$, $M_{2}$ and $M$ be in ${\mathcal F}_{(\g,\g_{\0})}$. Then, writing
$W(\e)$ for the pseudoreflection group associated with $\e$,
\begin{itemize}
\item[(a)] ${\mathcal V}_{(\g,\g_{\0})}(M)\cong {\mathcal V}^{\operatorname{rank}}_{(\e,\e_{\0})}(M)/W(\e)$;
\item[(b)] ${\mathcal V}_{(\g,\g_{\0})}(M_{1}\otimes M_{2})=
{\mathcal V}_{(\g,\g_{\0})}(M_{1})\cap {\mathcal V}_{(\g,\g_{\0})}(M_{2})$.
\item[(c)] Let $X$ be a conical subvariety of ${\mathcal V}_{(\g,\g_{\0})}({\mathbb C})$. Then
there exists $L$ in ${\mathcal F}$ with $X={\mathcal V}_{(\g,\g_{\0})}(L)$.
\item[(d)] If $M$ is indecomposable then $\operatorname{Proj}({\mathcal V}_{(\g,\g_{\0})}(M))$
is connected.
\end{itemize}
\end{thm}

\begin{proof} (a) This statement follows from the realization of ${\mathcal V}_{(\e,\e_{\0})}(M)$
as a rank variety and Theorem~\ref{T:isosupports}(b). (b) This follows by using part (a) and
the fact that the tensor product property for support varieties holds for
${\mathcal V}^{\operatorname{rank}}_{(\e,\e_{\0})}(-)$. Parts (c) and (d) are proved using the
same arguments as those for support varieties of finite groups (cf. \cite{Car:83,Car:84}).
\end{proof}

We remark that since the stable module category of ${\mathcal F}_{(\g,\g_{\0})}$ is a symmetric monoidal
tensor category, one can consider the spectrum ``$\text{Spc}(\text{Stab }{\mathcal F}_{(\g,\g_{\0})})$'' as in
\cite{Bal}. Our results show that the
assignment $(-)\rightarrow {\mathcal V}_{(\g,\g_{\0})}(-)$ satisfies the properties
stated in Balmer's paper for support datum.

Using Theorem~\ref{T:supportquotient2} and the same arguments as in Theorem~\ref{T:supportprop}, we can
similarly realize the
support varieties for the Cartan type Lie superalgebras $W(n)$ and $S(n)$, and prove that they satisfy the
tensor product property.

\begin{thm} \label{T:supportprop2} Let ${\mathfrak g}=W(n)$ or $S(n)$ and let $M_{1}$, $M_{2}$, $M$
be finite dimensional modules in ${\mathcal C}_{(\g,\g_{0})}$.
\begin{itemize}
\item[(a)] ${\mathcal V}_{(\g,\g_{0})}(M)\cong {\mathcal V}^{\operatorname{rank}}_{(\widetilde{\f},\widetilde{\f}_{0})}(M)/[W\ltimes T]$;
\item[(b)] ${\mathcal V}_{(\g,\g_{0})}(M_{1}\otimes M_{2})=
{\mathcal V}_{(\g,\g_{0})}(M_{1})\cap {\mathcal V}_{(\g,\g_{0})}(M_{2})$.
\item[(c)] Let $X$ be a conical subvariety of ${\mathcal V}_{(\g,\g_{0})}({\mathbb C})$. Then
there exists $L$ in ${\mathcal F}$ with $X={\mathcal V}_{(\g,\g_{0})}(L)$.
\item[(d)] If $M$ is indecomposable then $\operatorname{Proj}({\mathcal V}_{(\g,\g_{0})}(M))$
is connected.
\end{itemize}
\end{thm}

\subsection{Connections with Atypicality} Let $\g$ be a basic classical Lie superalgebra
with a non-degenerate invariant supersymmetric even bilinear form $(-,-)$. 
For a weight $\lambda\in {\mathfrak t}^{*}$,  the
\emph{atypicality} of $\lambda$ is the maximal number of linearily independent, mutually orthogonal,
positive isotropic roots $\alpha \in  \Phi^{+}$ such that $ (\lambda+\rho,\alpha)=0$ where
where $\rho=\frac{1}{2}(\sum_{\alpha\in \Phi_{\0}^{+}} \alpha-\sum_{\alpha\in \Phi_{\1}^{+}} \alpha)$.

For a simple finite dimensional ${\mathfrak g}$-supermodule $L(\lambda)$ with
highest weight $\lambda$, the atypicality of $L(\lambda)$ is defined to be $\text{atyp}(L(\lambda)):=
\text{atyp}(\lambda)$. In \cite[Conjecture 7.2.1]{BKN1}, a conjecture was stated relating the
atypicality of $L(\lambda)$ to the dimension of ${\mathcal V}_{(\e,\e_{\0})}(L(\lambda))$. This
conjecture was verified for ${\mathfrak g}=\mathfrak{gl}(m|n)$ in \cite{BKN2}. In light of
the results of the previous section, it seems reasonable to modify the ``Aypicality Conjecture''
to a statement which does not involve detecting subalgebras.

\begin{conj}\label{C:atypconjecture} Let $\g$ be a simple basic classical Lie superalgebra
and let $L(\lambda)$ be a finite dimensional simple $\g$-supermodule.
Then $$\operatorname{atyp}(L(\lambda))=\dim {\mathcal V}_{(\g,\g_{\0})}(L(\lambda)).$$
\end{conj}

In \cite{BKN2} this modified version of the conjecture has been also verified, and the support
varieties for the simple modules have been completely described. This leads one to believe that the
atypicality for any classical Lie superalgebra ${\mathfrak g}$ should be defined for all modules $M$
in ${\mathcal F}_{(\g,\g_{\0})}$ (in a functorial way) as the dimension of ${\mathcal V}_{(\g,\g_{\0})}(M)$.

For the Lie superalgebras of Cartan type, Serganova \cite{Se} defined the notion of typical and atypical
weights.  In \cite[Theorem 7.3.1]{BaKN}, it was shown that for typical weights the support varieties for
simple $W(n)$-modules is $\{0\}$, and for atypical weights the support varieties are equal to
${\mathcal V}_{(\g,\g_{0})}({\mathbb C})$. In this case it also makes sense to define the atypicality
of a finite dimensional module in ${\mathcal C}_{(\g,\g_{0})}$ as $\dim{\mathcal V}_{(\g,\g_{0})}(M)$.

\end{document}